\documentclass[12pt,letterpaper]{article}
\usepackage[latin1]{inputenc}
\usepackage[all]{xy}
\usepackage{bbm}
\usepackage{pifont}
\usepackage{mathrsfs}
\usepackage{amsfonts}
\usepackage{bm}
\usepackage{amsmath,amssymb,amsthm,stmaryrd}
\usepackage{relsize}
\usepackage[nottoc,notlot,notlof]{tocbibind}
\usepackage{mdwlist}
\usepackage{natbib}
\usepackage{hyperref}
\usepackage{authblk}
\theoremstyle{plain}
\newtheorem{teorema}{Theorem}[section]

\newtheorem{lema}{Lemma}[section]
\newtheorem{proposicion}[teorema]{Proposition}
\theoremstyle{definition}

\newtheorem{definicion}{Definition}[section]

\newtheorem{observacion}{Observation}[section]

\newcommand{\ob}{\mathrm{Ob}}

\newcommand{\uno}{\mathbf{1}}

\newcommand{\mnd}{\mathbf{Mnd}}
\newcommand{\emo}{\mathbf{EM}}
\newcommand{\kle}{\mathbf{Kl}}
\newcommand{\EM}{\mathrm{EM}}
\newcommand{\Kl}{\mathrm{Kl}}
\newcommand{\cat}{\mathbf{CAT}}
\newcommand{\Cat}{\mathbf{Cat}}

\newcommand*{\email}[1]{%
    \small\href{mailto:#1}{#1}\par
    }
\xyoption{2cell}
\xyoption{curve}
\xyoption{rotate}
\UseAllTwocells

\title{Another characterization of no-iteration distributive laws}
\author{\normalsize Enrique Ruiz Hernández}
\affil{\normalsize Centro de Investigación en Teoría de Categorías y sus Aplicaciones\\ \small\email{e.ruiz-hernandez@cinvcat.org.mx}}
\date{}

\begin{document}
\catcode`\"=12\relax
\maketitle
\begin{abstract}
We provide a characterization of no-iteration distributive laws in terms of its monads in extensive form only. To do that, it is necessary to take account of both right and left extension systems. We also give, in this right-left perspective, characterizations of the 1-cells and 2-cells in $\emo(K)$ and $\kle(K)$.
\end{abstract}

\section{Introduction}

\noindent In \citet{MW2010}, it is established that a distributive law in a 2-category can alternatively be defined in extensive form (dispensing with the iterates $tt$ and $ttt$ of the usual definition of a distributive law) as in the monad case (see \citet{eM1976}, Exercise 1.3.12). However, the result is (rather ``would be'' since a definition is not given; see Theorem~\ref{T:bijdistlawalg}) a definition in terms of the algebras of one of the monads and in terms of the other monad, both in extensive form. So we wondered if it was possible to have a definition of a no-iteration distributive law not using the algebras (as it is done in \citet{MW2010}), but instead, using directly the extensive form of the monads involved in the distributive law.

Such a definition is possible: we need take account of both right and left extension systems.

When we speak of the 1-cells and the 2-cells in $\emo(K)$ and $\kle(K)$, where $K$ is a 2-category, we shall follow the nomenclature in \citet{CS2010}. They denote $\emo(\Cat)$ and $\kle(\Cat)$ as $\mnd_\EM$ and $\mnd_\Kl$, resp.

In section~\ref{S:1}, we present the prerequisites in order to provide a clear framework for this the paper. In section~\ref{S:2}, we treat in this right-left perspective the 1-cells and 2-cells in $\emo(K)$ and in $\kle(K)$. In section~\ref{S:3}, we do the same for distributive laws; this is where we give the sought definition (characterization). In this section we also give a definition adapted to the case $K=\Cat$, and finally we provide a characterization in between that we think could be useful to produce a more tractable definition for the pseudo case and extend the results in \citet{MW2013}

\section{Monads}\label{S:1}

\begin{definicion}
Let $A$ be a category. A monad on $A$ in monoidal form consists of
\begin{enumerate*}
 \item[(i)] a functor $T:A\rightarrow A$,
 \item[(ii)] a natural transformation $\eta:1_A\Rightarrow T$,
 \item[(iii)] a natural transformation $\mu:TT\Rightarrow T$ 
\end{enumerate*}
such that the following diagrams commute:
$$\xymatrix{
T\ar[r]^{\eta T}\ar[dr]_1 & TT\ar[d]^\mu & T\ar[l]_{T\eta}\ar[dl]^1 & &
TTT\ar[r]^{T\mu}\ar[d]_{\mu T} & TT\ar[d]^\mu\\
& T &, & &
TT\ar[r]_\mu & T.
}$$
\end{definicion}

\begin{definicion}[\citet{eM1976}]\label{D:extensiveform}
A \textit{monad in extensive form on a category $A$} consists of
\begin{enumerate*}
 \item[(i)] a map $T:\ob A\rightarrow\ob A$,
 \item[(ii)] a family of arrows $\eta a:a\rightarrow Ta$ with $a\in A$,
 \item[(iii)] a family of maps $(-)^\mathbbm{T}_{a,b}:A(a,Tb)\rightarrow A(Ta,Tb)$
\end{enumerate*}
such that
\begin{enumerate*}
 \item $(\eta a)^\mathbbm{T}=1_{Ta}$,
 \item for every $f:a\rightarrow Tb$ the diagram
 $$\xymatrix{
 a\ar[r]^{\eta a}\ar[dr]_f & Ta\ar[d]^{f^\mathbbm{T}}\\
 & Tb
 }$$
 commutes,
 \item for every $f:a\rightarrow Tb$ and every $g:b\rightarrow Tc$ the diagram
 $$\xymatrix{
 Ta\ar[r]^{f^\mathbbm{T}}\ar[dr]_{(g^\mathbbm{T}f)^\mathbbm{T}} & Tb\ar[d]^{g^\mathbbm{T}}\\
  & Tc
 }$$
 commutes.
\end{enumerate*}

We call $(-)^\mathbbm{T}$ the \textit{extension operation of $T$}.
\end{definicion}

\begin{proposicion}
Let $A$ be a category. There is a bijection between the monads $(T,\eta,\mu)$ on $A$ in monoidal form and the monads $(T,\eta,(-)^\mathbbm{T})$ on $A$ in extensive form.
\end{proposicion}
\begin{proof}
Let $(T,\eta,\mu)$ be a monad on $A$ in monoidal form. Let $f:a\rightarrow Tb$ be an arrow in $A$. Define
$$f^\mathbbm{T}:=\mu b\circ Tf.$$

Conversely, let $(T,\eta,(-)^\mathbbm{T})$ be a monad on $A$ in extensive form. Let $f:a\rightarrow b$ be in $A$, and define
$$Tf:=(\eta b\circ f)^\mathbbm{T}\qquad\text{and}\qquad\mu a:=(1_{Ta})^\mathbbm{T}.$$
\end{proof}

In \citet{MW2010}, this situation is generalized, and there an extension system is defined, so that a monad in a 2-category may be defined alternatively in terms of an extension system. What we define as a left pasting operator is called \textit{a pasting operator} in \citet{MW2010}; the same for a left extension system. This change of nomenclature will be clearly justified below. Neither right pasting operators nor right extension systems are mentioned there.

\begin{definicion}[\citet{MW2010}]
Let $K$ be a 2-category. A \textit{left pasting operator}
$$(-)^\#:K(x,b)(1,s)\rightarrow K(x,c)(t,u)$$
is a family of functions
$$(-)^\#_{f,g}:K(x,b)(f,sg)\rightarrow K(x,c)(tf,ug)$$
that respects whiskering and blistering; i. e., $\vartheta^\# h=(\vartheta h)^\#$ and $(\vartheta\cdot\kappa)^\#=\vartheta^\#\cdot t\kappa$, where whiskering is
$$
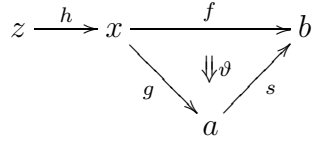
\xymatrix{
z\ar[r]^h & x\ar[rr]^f\ar[dr]_g\rrtwocell<\omit>{<3>\vartheta} & & b\\
 & & a\ar[ur]_s &
}$$
and blistering is
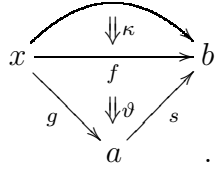
$$\xymatrix{
x\ar[rr]_f\ar[dr]_g\rrtwocell<\omit>{<4>\vartheta}\ar@/^1.7pc/[rr]\rrtwocell<\omit>{<-2>\kappa} & & b\\
& a\ar[ur]_s & .
}$$

Dually, a \textit{right pasting operator}
$$(-)^\#:K(b,y)(1,t)\rightarrow K(a,y)(s,u)$$
is a family of functions
$$(-)^\#_{f,g}:K(b,y)(f,gt)\rightarrow K(a,y)(fs,gu)$$
that respects whiskering and blistering; i. e., $h\vartheta^\#=(h\vartheta)^\#$ and $(\vartheta\cdot\kappa)^\#=\vartheta^\#\cdot\kappa s$, where whiskering is
$$\xymatrix{
b\ar[rr]^f\ar[dr]_t\rrtwocell<\omit>{<3>\vartheta} & & y\ar[r]^h & z\\
& c\ar[ur]_g &
}$$
and blistering is
$$\xymatrix{
b\ar[rr]_f\ar[dr]_t\rrtwocell<\omit>{<4>\vartheta}\ar@/^1.7pc/[rr]\rrtwocell<\omit>{<-2>\kappa} & & y\\
& c\ar[ur]_g & .
}$$
\end{definicion}

\begin{lema}[\citet{MW2010}]\label{L:2cellslandrpastingoperators}
Let $K$ be a 2-category. For arrows in $K$ configured as in the following diagram
$$
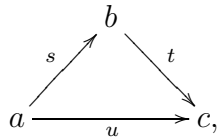
\xymatrix{
& b\ar[dr]^t &\\
a\ar[ur]^s\ar[rr]_u & & c,
}$$
left pasting operators
$$K(x,b)(1,s)\rightarrow K(x,c)(t,u)$$
are in bijective correspondence with 2-cells $ts\Rightarrow u$.

Dually, right pasting operators
$$(-)^\#:K(b,y)(1,t)\rightarrow K(a,y)(s,u)$$
are in bijective correspondence with 2-cells $ts\Rightarrow u$.
\end{lema}
\begin{proof}
Given a left pasting operator
$$(-)^\#:K(x,b)(1,s)\rightarrow K(x,c)(t,u),$$
we have the 2-cell
$$(1_s)^\#_{s,1_a}:ts\Rightarrow u.$$
Any 2-cell $\vartheta:f\Rightarrow sg:x\rightarrow b$ arises by whiskering $1_s$ at $x$ by $g$ and by blistering the result at $sg$ by $\vartheta$; i. e.,
\begin{equation}\label{E:whiskblist}
\vartheta^\#=(1_s)^\#g\cdot t\vartheta;
\end{equation}
diagrammatically,
$$\xymatrix{
x\ar[rr]^f\ar[dr]_g\rrtwocell<\omit>{<3>\vartheta} & & b\ar@{}|{=}[r] &
x\ar[r]_g\ar@/^2pc/[rrr]^f & a\ar[rr]^s\ar[dr]_{1_a}\rrtwocell<\omit>{<3>1_s}\rtwocell<\omit>{<-2.5>\vartheta} & & b \\
& a\ar[ur]_s & &
& & a\ar[ur]_s &.
}$$
That is, every $K(x,b)(1,s)\rightarrow K(x,c)(t,u)$ is completely determined by $(1_s)^\#_{s,1_a}$, and the latter can be any 2-cell $ts\Rightarrow u$. It follows that the assignment $(-)^\#\mapsto(1_s)^\#_{s,1_a}$ is a bijection.

Dually, the assignment $(-)^\#\mapsto(1_t)^\#_{t,1_c}$ for a right pasting operator $(-)^\#$ is a bijection, whence we get the other correspondence. The dual equation to \eqref{E:whiskblist} is
\begin{equation}\label{E:whiskblistright}
\varkappa^\#=k(1_t)^\#\cdot\varkappa s
\end{equation}
for $\varkappa:h\Rightarrow kt:b\rightarrow y$.
\end{proof}

\begin{definicion}[\citet{MW2010}]
Let $K$ be a 2-category and $a\in K$ be an object of $K$. A \textit{left extension system on $a$} consists of an arrow $s:a\rightarrow a$, a 2-cell $\eta:1_a\Rightarrow s$ and a left pasting operator
$$(-)^\mathbbm{s}:K(x,a)(1,s)\rightarrow K(x,a)(s,s),$$
that we call the \textit{left $\mathbbm{s}$-extension operator}, such that
\begin{enumerate*}
 \item $\eta^\mathbbm{s}=1_s$,
 \item for every $\vartheta:f\Rightarrow sg:x\rightarrow a$ the diagram
 $$\xymatrix{
 f\ar[r]^{\eta f}\ar[dr]_\vartheta & sf\ar[d]^{\vartheta^\mathbbm{s}}\\
 & sg
 }$$
 commutes,
 \item for every $\vartheta:f\Rightarrow sg:x\rightarrow a$ and every $\kappa:g\Rightarrow sh:x\rightarrow a$ the diagram
 $$\xymatrix{
 sf\ar[r]^{\vartheta^\mathbbm{s}}\ar[dr]_{(\kappa^\mathbbm{s}\vartheta)^\mathbbm{s}} & sg\ar[d]^{\kappa^\mathbbm{s}}\\
 & sh
 }$$
 commutes.
\end{enumerate*}

A \textit{right extension system on $a$} consists of an arrow $s:a\rightarrow a$, a 2-cell $\eta:1_a\Rightarrow s$ and a right pasting operator
$$(-)^\mathbbm{s}:K(a,y)(1,s)\rightarrow K(a,y)(s,s),$$
that we call the \textit{right $\mathbbm{s}$-extension operator}, such that
\begin{enumerate*}
 \item $\eta^\mathbbm{s}=1_s$,
 \item for every $\vartheta:f\Rightarrow gs:a\rightarrow y$ the diagram
 $$\xymatrix{
 f\ar[r]^{f\eta}\ar[dr]_\vartheta & fs\ar[d]^{\vartheta^\mathbbm{s}}\\
 & gs
 }$$
 commutes,
 \item for every $\vartheta:f\Rightarrow gs:a\rightarrow y$ and every $\kappa:g\Rightarrow hs:a\rightarrow y$ the diagram
 $$\xymatrix{
 fs\ar[r]^{\vartheta^\mathbbm{s}}\ar[dr]_{(\kappa^\mathbbm{s}\vartheta)^\mathbbm{s}} & gs\ar[d]^{\kappa^\mathbbm{s}}\\
 & hs
 }$$
 commutes.
\end{enumerate*}
\end{definicion}

\begin{teorema}[\citet{MW2010}]
For a 2-cell $\eta:1_a\Rightarrow s:a\rightarrow a$ in a 2-category $K$, there is a bijection between (right) left extension systems $(s,\eta,(-)^\mathbbm{s})$ and monads $(s,\eta,\mu)$.
\end{teorema}
\begin{proof}
By Lemma~\ref{L:2cellslandrpastingoperators}, there is a bijection between (right) left pasting operators $(-)^\mathbbm{s}$ and 2-cells $\mu:ss\Rightarrow s$. So, let $(s,\eta,\mu)$ be a monad on $a$ in $K$. The correspondence in Lemma~\ref{L:2cellslandrpastingoperators} yields
$$\vartheta^\mathbbm{s}:=\mu g\cdot s\vartheta$$
for $\vartheta:f\Rightarrow sg:x\rightarrow a$ and
$$\varkappa^\mathbbm{s}:=k\mu\cdot\varkappa s$$
for $\varkappa:h\Rightarrow sk:a\rightarrow y$.

Conversely, let $(s,\eta,(-)^\mathbbm{s})$ be a (right) left extension system on $a$ in $K$. Then, Lemma~\ref{L:2cellslandrpastingoperators} yields
$$\mu:=(1_s)^\mathbbm{s}.$$
\end{proof}

In \citet{MRW2002}, there is a theorem (Proposition 3.5) establishing a bijective correspondence between distributive laws of a monad $s$ over a monad $t$ in a 2-category $K$ and $s$-algebras $\alpha:sts\rightarrow s$ satisfying the commutativity of certain diagrams. More precisely:

\begin{proposicion}[\citet{MRW2002}]
Given two monads $(s,\eta',\mu')$ and $(t,\eta,\mu)$ on $a$ in a 2-category $K$, there is a bijective correspondence between distributive laws $\lambda:st\Rightarrow ts$ of $(s,\eta',\mu')$ over $(t,\eta,\mu)$ and $s$-algebras $\alpha:sts\Rightarrow ts$ that satisfy the commutativity of the following diagrams:
$$\xymatrix{
sts^2\ar[r]^{st\mu'}\ar[d]_{\alpha s} & sts\ar[d]^\alpha &
s^2\ar[r]^{s\eta s}\ar[d]_{\mu'} & sts\ar[d]^\alpha &
st^2s\ar[r]^{st\eta'ts}\ar[d]_{s\mu s} & ststs\ar[r]^{\alpha ts} & tsts\ar[r]^{t\alpha} & t^2s\ar[d]^{\mu s}\\
ts^2\ar[r]_{t\mu'} & ts &
s\ar[r]_{\eta s} & ts &
sts\ar[rrr]_\alpha & & & ts
}$$
given by
$$\xymatrix{
\lambda\ar@{|->}[r] & sts\ar[r]^{\lambda s} & ts^2\ar[r]^{t\mu'} & ts
}$$
and inverse given by
$$\xymatrix{
\alpha\ar@{|->}[r] & st\ar[r]^{st\eta'} & sts\ar[r]^\alpha & ts.
}$$
\end{proposicion}

In \citet{MW2010}, algebras for a monad $(s,\eta',(-)^\mathbbm{s})$ in extensive form are defined too; they are also in extensive form. It is shown there that the category of algebras with domain $x$ of $(s,\eta',(-)^\mathbbm{s})$ and the category of algebras with domain $x$ of the monad $(s,\eta',\mu')$ corresponding to $(s,\eta',(-)^\mathbbm{s})$ are isomorphic.

Using this and the previous theorem, \citet{MW2010} gives a theorem establishing a bijective correspondence:
\begin{teorema}[\cite{MW2010}]\label{T:bijdistlawalg}
Let $K$ be a 2-category, and $a\in K$. Let $(s,\eta',(-)^\mathbbm{s})$ and $(t,\eta,(-)^\mathbbm{t})$ be left extension systems on $a$. A distributive law $\lambda$ of $(s,\eta',\mu')$ over $(t,\eta,\mu)$,  the corresponding monads in monoidal form to $s$ and $t$, resp., can be given as an $\mathbbm{s}$-algebra $(ts,(-)^{\bm{\lambda}})$ such that for every $\gamma:g\Rightarrow sh:x\rightarrow a$ and $\vartheta:f\Rightarrow tsg:x\rightarrow a$ the diagram
$$\xymatrix{
sf\ar[r]^{\vartheta^{\bm{\lambda}}}\ar[dr]_{(t\gamma^\mathbbm{s}\cdot\vartheta)^{\bm{\lambda}}} & tsg\ar[d]^{t\gamma^\mathbbm{s}} \\
& tsh
}$$
commutes,
$$(t\eta'\cdot\eta)^{\bm{\lambda}}=\eta s$$
and for every $\zeta:k\Rightarrow tsu:y\rightarrow a$ and every $\kappa:u\Rightarrow tsv:y\rightarrow a$ the diagram
$$\xymatrix{
sk\ar[r]^{\zeta^{\bm{\lambda}}}\ar[dr]_{((\kappa^{\bm{\lambda}})^\mathbbm{t}\cdot\zeta)^{\bm{\lambda}}} & tsu\ar[d]^{(\kappa^{\bm{\lambda}})^\mathbbm{t}} \\
& tsv
}$$
commutes.
\end{teorema}

\section{No-iteration monad morphisms}\label{S:2}

\noindent Proposition 2.2.6 in \citet{MM2007} provides a bijective correspondence involving monad $\Kl$-morphims:

\begin{proposicion}[\citet{MM2007}]
Kleisli liftings $\bar{F}:C_H\rightarrow D_K$ are in bijective correspondence with families $\lambda_a:FHa\rightarrow KFa$ satisfying the commutativity of
$$\xymatrix{
& FHa\ar[dd]^{\lambda_a} \\
Fa\ar[ur]^{F\eta_Ha}\ar[dr]_{\eta_KFa} & \\
& KFa
}$$
and of
$$\xymatrix{
FHa\ar[r]^{Ff^\mathbbm{H}}\ar[d]_{\lambda_a} & FHb\ar[d]^{\lambda_b} \\
KFa\ar[r]_{(\lambda_b\circ f)^\mathbbm{K}} & KFb
}$$
for $f:a\rightarrow Hb$, where $(-)^\mathbbm{H}$ is the extension operation of the monad $H$ and $(-)^\mathbbm{K}$ is the extension operation of the monad $K$.
\end{proposicion}

If we paraphrase it in the context of a 2-category, we get the following.

\begin{proposicion}
Let $(c,s,\eta',\mu')$ and $(d,t,\eta,\mu)$ be two monads in a 2-category $K$. Let us call the triangle in the axioms for a monad $\Kl$-morphism \textit{$\Kl$-compatibility for units (KlU)}
$$\xymatrix{
& ft\ar[dd]^\kappa \\
f\ar[ur]^{f\eta}\ar[dr]_{\eta'f} & \\
& sf
}$$
and call the pentagon in those axioms \textit{$\Kl$-compatibility for multiplications (KlM)}
$$\xymatrix{
ftt\ar[r]^{\kappa t}\ar[d]_{f\mu} & sft\ar[r]^{s\kappa} & ssf\ar[d]^{\mu'f}\\
ft\ar[rr]_\kappa & & sf.
}$$
Then, there is a bijective correspondence between 2-cells $\kappa:ft\Rightarrow sf$ satisfying KlU, KlM and 2-cells $\kappa:ft\Rightarrow sf$ satisfying KlU and the commutativity of the diagram
$$\xymatrix{
ftg\ar[r]^{f\vartheta^{\mathbbm{t}}}\ar[d]_{\kappa g} & fth\ar[d]^{\kappa h} \\
sfg\ar[r]_{(\kappa h\cdot f\vartheta)^\mathbbm{s}} & sfh 
}$$
for every 2-cell $\vartheta:g\Rightarrow th:x\rightarrow d$ in $K$, where $(-)^\mathbbm{t}$ and $(-)^\mathbbm{s}$ are the left pasting operators corresponding to the monads $(d,t,\eta,\mu)$ and $(c,s,\eta',\mu')$, resp.
\end{proposicion}

\begin{definicion}
Given a 2-category $K$ and a 2-cell $\kappa:ft\Rightarrow sf$ in $K$, the 2-cell is called a \textit{no-iteration monad $\Kl$-morphism $(f,\kappa)$ from $(-)^\mathbbm{t}$ to $(-)^\mathbbm{s}$} if $\kappa$ satisfies KlU and the commutativity of the last diagram in the previous proposition. 
\end{definicion}

For $K=\cat$, we have the following definition for no-iteration monad $\Kl$-morphisms from a monad $T$ to a monad $S$ since clearly there is a bijection between these and the monad $\Kl$-morphisms from $T$ to $S$; this fact is Proposition 2.2.6 in \citet{MM2007}.

\begin{definicion}
Let $(S,\eta',(-)^\mathbbm{S})$ be a monad on $B$ in extensive form and let $(T,\eta,(-)^\mathbbm{T})$ be a monad on $A$ in extensive form. A \textit{no-iteration monad $\Kl$-morphism} (or a \textit{monad $\Kl$-morphism in extensive form}) \textit{from $(T,\eta,(-)^\mathbbm{T})$ to $(S,\eta',(-)^\mathbbm{S})$} consists of a functor $F:A\rightarrow B$ and a family of arrows $\kappa a:FTa\rightarrow SFa$ such that
\begin{enumerate*}
 \item for every $a\in A$ the diagram
 $$\xymatrix{
 & FTa\ar[dd]^{\kappa a} & \\
 Fa\ar[ur]^{F\eta a}\ar[dr]_{\eta' Fa} & & \\
 & SFa &
 }$$
 commutes,
 \item for every $f:a\rightarrow Tb$ in $A$ the diagram
 $$\xymatrix{
 FTa\ar[rr]^{Ff^\mathbbm{T}}\ar[d]_{\kappa a} & & FTb\ar[d]^{\kappa b}\\
 SFa\ar[rr]_{(\kappa b\circ Ff)^\mathbbm{S}} & & SFb
 }$$
 commutes.
\end{enumerate*}
\end{definicion}

We have the analog for monad $\EM$-morphisms in an arbitrary 2-category.

\begin{proposicion}
Let $(c,s,\eta',\mu')$ and $(d,t,\eta,\mu)$ be two monads in a 2-category $K$. Let us call the triangle in the axioms for a monad $\EM$-morphism \textit{$\EM$-compatibility for units (EMU)}
$$\xymatrix{
& sf\ar[dd]^\varphi \\
f\ar[ur]^{\eta'f}\ar[dr]_{f\eta} & \\
& ft
}$$
and call the pentagon in those axioms \textit{$\EM$-compatibility for multiplications (EMM)}
$$\xymatrix{
ssf\ar[r]^{s\varphi}\ar[d]_{\mu'f} & sft\ar[r]^{\varphi t} & ftt\ar[d]^{f\mu}\\
sf\ar[rr]_\varphi & & ft.
}$$
Then, there is a bijective correspondence between 2-cells $\varphi:sf\Rightarrow ft$ satisfying EMU, EMM and 2-cells $\varphi:sf\Rightarrow ft$ satisfying EMU and the commutativity of the diagram
$$\xymatrix{
gsf\ar[r]^{\vartheta^\mathbbm{s}f}\ar[d]_{g\varphi} & hsf\ar[d]^{h\varphi} \\
gft\ar[r]_{(h\varphi\cdot\vartheta f)^\mathbbm{t}} & hft 
}$$
for every 2-cell $\vartheta:g\Rightarrow hs:c\rightarrow y$ in $K$, where $(-)^\mathbbm{t}$ and $(-)^\mathbbm{s}$ are the right pasting operators corresponding to the monads $(d,t,\eta,\mu)$ and $(c,s,\eta',\mu')$, resp.
\end{proposicion}
\begin{proof}
Let $\varphi:sf\Rightarrow tf$ be a 2-cell in $K$ satisfying EMU and EMM. Let $\vartheta:g\Rightarrow hs:c\rightarrow y$ be a 2-cell in $K$. Then,
$$\xymatrix{
& & c\ar[dr]^g & \\
& c\ar[dr]^s\ar[ur]^s\rrtwocell<\omit>{\ \vartheta^\mathbbm{s}} & & y\ar@{}|{=}[dr] & \\
d\ar[ur]^f\ar[dr]_t\rrtwocell<\omit>{\varphi} & & c\ar[ur]_h & & \\
& d\ar[ur]_f & &
}$$
$$\xymatrix{
& & c\ar[dr]^s\ar[rr]^g\rrtwocell<\omit>{<3>\vartheta} & & y \\
d\ar[r]^f\ar[dr]_t & c\ar[rr]_s\ar[ur]^s\rrtwocell<\omit>{<-3>\,\mu'} & & c\ar[ur]_h & \ar@{}|{=}[r] & \\
\rrtwocell<\omit>{<-4>\varphi} & d\ar[urr]_f & & &
}$$
$$\xymatrix{
& c\ar[dr]^s & & & \\
d\ar[ur]^f\ar[dr]_t\rrtwocell<\omit>{\varphi}\ar@/_2.3pc/[ddrr]_t\ddrrtwocell<\omit>{<3>\mu} & & c\ar[rr]^g\ar[dr]_s\rrtwocell<\omit>{<3>\vartheta} & & y\ar@{}|{=}[r] & \\
& d\ar[ur]_f\ar[dr]_t\rrtwocell<\omit>{\varphi} & & c\ar[ur]_h & \\
& & d\ar[ur]_f & &
}$$
$$\xymatrix{
& c\ar[dr]^s & & & \\
d\ar[ur]^f\ar[r]_t\rrtwocell<\omit>{<-3>\varphi}\ar[dr]_t & d\ar[r]_f\rtwocell<\omit>{<4>\qquad\ (h\varphi\cdot\vartheta f)^\mathbbm{t}} & c\ar[rr]^g & & y \\
& d\ar[rr]_f & & c\ar[ur]_h & .
}$$

Conversely, suppose $\varphi:sf\Rightarrow ft$ is a 2-cell in $K$ satisfying EMU and the commutativity of the last diagram in the proposition.

If $\vartheta=1_s$ in the sequence of equations above, then we get the result.
\end{proof}

\begin{definicion}
Given a 2-category $K$ and a 2-cell $\varphi:sf\Rightarrow ft$ in $K$, the 2-cell is called a \textit{no-iteration monad $\EM$-morphism $(f,\varphi)$ from $(-)^\mathbbm{t}$ to $(-)^\mathbbm{s}$} if $\varphi$ satisfies EMU and the commutativity of the last diagram in the previous proposition.
\end{definicion}

\begin{proposicion}
Let $(c,s,\eta',\mu')$ and $(d,t,\eta,\mu)$ be two monads in a 2-category $K$. Let $(-)^\mathbbm{t}$ and $(-)^\mathbbm{s}$ be the left pasting operators corresponding to the monads $(d,t,\eta,\mu)$ and $(c,s,\eta',\mu')$, resp. Let $\varphi:ft\Rightarrow sf$ and $\varphi':f't\Rightarrow sf'$ be 2-cells in $K$. Then, there is bijective correspondence between 2-cells $\chi:f'\Rightarrow sf$ satisfying the commutativity of the diagram
$$\xymatrix{
f't\ar[r]^{\varphi'}\ar[d]_{\chi t} & sf'\ar[r]^{s\chi} & ssf\ar[d]^{\mu'f} \\
sft\ar[r]_{s\varphi} & ssf\ar[r]_{\mu'f} & sf
}$$
and 2-cells $\chi:f'\Rightarrow sf$ satisfying the commutativity of the diagram
\begin{equation}\label{D:noit2cellKl}
\xymatrix{
f't\ar[r]^{\varphi'}\ar[d]_{\chi t} & sf\ar[d]^{\chi^\mathbbm{s}} \\
sft\ar[r]_{\varphi^{\mathbbm{s}}} & sf.
}
\end{equation}
\end{proposicion}

\begin{definicion}
Let $(c,s,\eta',\mu')$ and $(d,t,\eta,\mu)$ be two monads in a 2-category $K$. Let $(-)^\mathbbm{t}$ and $(-)^\mathbbm{s}$ be the left pasting operators corresponding to the monads $(d,t,\eta,\mu)$ and $(c,s,\eta',\mu')$, resp. Let $\varphi:ft\Rightarrow sf$ and $\varphi':f't\Rightarrow sf'$ be no-iteration monad $\Kl$-morphisms. A 2-cell $\chi:f'\Rightarrow sf$ in $K$ satisfying \eqref{D:noit2cellKl} is called a \textit{no-iteration $\Kl$-transformation from $(f,\varphi)$ to $(f',\varphi')$}.
\end{definicion}

\begin{proposicion}
Let $(c,s,\eta',\mu')$ and $(d,t,\eta,\mu)$ be two monads in a 2-category $K$. Let $(-)^\mathbbm{t}$ and $(-)^\mathbbm{s}$ be the right pasting operators corresponding to the monads $(d,t,\eta,\mu)$ and $(c,s,\eta',\mu')$, resp. Let $\varphi:sf\Rightarrow ft$ and $\varphi':sf'\Rightarrow f't$ be 2-cells in $K$. Then, there is bijective correspondence between 2-cells $\varrho:f\Rightarrow f't$ satisfying the commutativity of the diagram
$$\xymatrix{
sf\ar[r]^\varphi\ar[d]_{s\varrho} & ft\ar[r]^{\varrho t} & f'tt\ar[d]^{f'\mu} \\
sf't\ar[r]_{\varphi't} & f'tt\ar[r]_{f'\mu} & f't
}$$
and 2-cells $\varrho:f\Rightarrow tf'$ satisfying the commutativity of the diagram
\begin{equation}\label{D:noit2cellEM}
\xymatrix{
sf\ar[r]^{\varphi}\ar[d]_{s\varrho} & ft\ar[d]^{\varrho^\mathbbm{t}} \\
sf't\ar[r]_{\varphi'^{\mathbbm{t}}} & f't.
}
\end{equation}
\end{proposicion}
\begin{proof}
We have
$$\xymatrix{
& c\ar[dr]_s & \\
d\ar[ur]^f\ar[dr]^t\ar@/_3.2pc/[drr]_t\rrtwocell<\omit>{\varphi}\drrtwocell<\omit>{<5.5>\mu} & & c\ar@{}|{=}[r] & \\
& d\ar[ur]^f\ar[r]^t\urtwocell<\omit>{<2>\varrho} & d\ar[u]_{f'}
}$$
$$\xymatrix{
& c\ar[dr]^s & \\
d\ar[ur]^f\ar[r]^t\urtwocell<\omit>{<2>\varrho}\ar@/_.7pc/[drr]_t\drtwocell<\omit>{<-2.5>\mu} & d\ar[u]_{f'}\ar[dr]^t\rtwocell<\omit>{\varphi'} & c \\
& & d.\ar[u]_{f'}
}$$
\end{proof}

\begin{definicion}
Let $(c,s,\eta',\mu')$ and $(d,t,\eta,\mu)$ be two monads in a 2-category $K$. Let $(-)^\mathbbm{t}$ and $(-)^\mathbbm{s}$ be the right pasting operators corresponding to the monads $(d,t,\eta,\mu)$ and $(c,s,\eta',\mu')$, resp. Let $\varphi:sf\Rightarrow ft$ and $\varphi':sf'\Rightarrow f't$ be no-iteration monad $\EM$-morphisms. A 2-cell $\varrho:f\Rightarrow f't$ in $K$ satisfying \eqref{D:noit2cellEM} is called a \textit{no-iteration $\EM$-transformation from $(f,\varphi)$ to $(f',\varphi')$}.
\end{definicion}

\section{No-iteration distributive laws}\label{S:3}

So there is another way to get a no-iteration version of a distributive law, one without passing through the algebras:

\begin{definicion}\label{D:noitdistlaw}
Let $K$ be a 2-category and $a\in K$. Let $(s,\eta',(-)^\mathbbm{s})$ be a right extension system on $a$ and $(t,\eta,(-)^\mathbbm{t})$ be a left extension system on $a$. A \textit{no-iteration distributive law} or a \textit{distributive law  in extensive form of $(s,\eta',(-)^\mathbbm{s})$ over $(t,\eta,(-)^\mathbbm{t})$} is a 2-cell $\lambda:st\Rightarrow ts$ such that
\begin{enumerate*}
 \item the diagrams
$$\xymatrix{
& st\ar[dd]^\lambda & &
& st\ar[dd]^\lambda\\
s\ar[ur]^{s\eta}\ar[dr]_{\eta s} & & &
t\ar[ur]^{\eta't}\ar[dr]_{t\eta'}\\
& ts & &
& ts
}$$
commute,
 \item for every 2-cell $\vartheta:f\Rightarrow tg:x\rightarrow a$ the diagram
 $$\xymatrix{
 stf\ar[rr]^{s\vartheta^\mathbbm{t}}\ar[d]_{\lambda f} & & stg\ar[d]^{\lambda g}\\
 tsf\ar[rr]_{(\lambda g\cdot s\vartheta)^\mathbbm{t}} & & tsg
 }$$
 commutes,
 \item for every 2-cell $\varkappa:h\Rightarrow ks:a\rightarrow y$ the diagram
 $$\xymatrix{
 hst\ar[rr]^{\varkappa^\mathbbm{s}t}\ar[d]_{h\lambda} & & kst\ar[d]^{k\lambda}\\
 hts\ar[rr]_{(k\lambda\cdot\varkappa t)^\mathbbm{s}} & & kts
 }$$
 commutes.
\end{enumerate*}
\end{definicion}

\begin{proposicion}\label{P:bijdistlawnnoitdistlaw}
Let $K$ be a 2-category and $a\in K$. There is a bijective correspondence between distributive laws of $(s,\eta',\mu')$ over $(t,\eta,\mu)$, monads on $a$, and no-iteration distributive laws of the right extension system $(s,\eta',(-)^\mathbbm{s})$ over the left extension system $(t,\eta,(-)^\mathbbm{t})$, where the latter correspond to $(s,\eta',\mu')$ and $(t,\eta,\mu)$, resp.
\end{proposicion}
\begin{proof}
Let $\lambda:st\Rightarrow ts$ be a distributive law of $(s,\eta',\mu')$ over $(t,\eta,\mu)$. Let $(s,\eta',(-)^\mathbbm{s})$ be the right extension system corresponding to $(s,\eta',\mu')$ and let $(t,\eta,(-)^\mathbbm{t})$ be the left extension system corresponding to $(t,\eta,\mu)$.

It is clear that Axiom 1 in Definition~\ref{D:noitdistlaw} holds.

Let $\vartheta:f\Rightarrow tg:x\rightarrow a$ be a 2-cell in $K$. Then, by compatibility of $\lambda$ with multiplication $\mu$,
$$\xymatrix{
x\ar[rr]^f\ar[dr]_g\rrtwocell<\omit>{<3>\vartheta} & & a\ar[rr]^t\ar[dr]_s &\drtwocell<\omit>{<.5>\lambda} & a\ar[dd]^s \\
& a\ar[ur]_t\ar[dr]_s\rrtwocell<\omit>{\lambda} & & a\ar[dr]^t &\ar@{}|{=}[r] &\\
& & a\ar[ur]_t\ar[rr]_t\rrtwocell<\omit>{<-3>\mu} & & a 
}$$
$$\xymatrix{
x\ar[rr]^f\ar[dr]_g\rrtwocell<\omit>{<3>\vartheta} & & a\ar[dr]^t &\\
& a\ar[ur]_t\ar[rr]_t\ar[d]_s\rrtwocell<\omit>{<-3>\mu}\drrtwocell<\omit>{\lambda} & & a\ar[d]^s \\
& a\ar[rr]_t & & a;
}$$
so Axiom 2 in Definition~\ref{D:noitdistlaw} holds for $\lambda$.

Likewise, by compatibility of $\lambda$ with $\mu'$, Axiom 3 holds.

Conversely, let $\lambda:st\Rightarrow ts$ be a 2-cell in $K$ such that it is a no-iteration distributive law of a right extension system $(s,\eta',(-)^\mathbbm{s})$ over a left extension system $(t,\eta,(-)^\mathbbm{t})$.

Trivially, $\lambda$ is compatible with the unit of $s$ and that of $t$. By Equation~\eqref{E:whiskblist},
$$\lambda^\mathbbm{t}=(1_t)^\mathbbm{t}s\cdot t\lambda;$$
so if $\vartheta=1_t$ in Axiom 2 for a no-iteration distributive law, then we have the compatibility of $\lambda$ with $(1_t)^\mathbbm{t}$, the multiplication of the monad $(t,\eta,(1_t)^\mathbbm{t})$.

Analogously, by Equation~\eqref{E:whiskblistright},
$$\lambda^\mathbbm{s}=t(1_s)^\mathbbm{s}\cdot\lambda s;$$
so if $\varkappa=1_s$ in Axiom 3 for a no-iteration distributive law, then we have the compatibility of $\lambda$ with $(1_s)^\mathbbm{s}$, the multiplication of the monad $(s,\eta',(1_s)^\mathbbm{s})$.
\end{proof}

Since $A\cong\Cat(\uno,A)$ for every category $A\in\Cat$, the previous definition for $\Cat$ becomes:

\begin{definicion}\label{D:noitdistlawcat}
Let $(S,\eta',(-)^\mathbbm{S})$ be a right extension system on a category $A$ and let $(T,\eta,(-)^\mathbbm{T})$ be a monad on $A$ in extensive form. A \textit{no-iteration distributive law} (or a \textit{distributive law in extensive form}) \textit{of $(S,\eta',(-)^\mathbbm{S})$ over $(T,\eta,(-)^\mathbbm{T})$} consists of a family of arrows $\lambda a:STa\rightarrow TSa$ such that
\begin{enumerate*}
 \item for every $a\in A$ the diagrams
 $$\xymatrix{
 & STa\ar[dd]^{\lambda a} & &
 & STa\ar[dd]^{\lambda a}\\
 Sa\ar[ur]^{S\eta a}\ar[dr]_{\eta Sa} & & &
 Ta\ar[ur]^{\eta'Ta}\ar[dr]_{T\eta'a} & \\
 & TSa & &
 & TSa
 }$$
 commute,
 \item for every $f:a\rightarrow Tb$ in $A$ the diagram
 $$\xymatrix{
 STa\ar[rr]^{Sf^\mathbbm{T}}\ar[d]_{\lambda a} & & STb\ar[d]^{\lambda b}\\
 TSa\ar[rr]_{(\lambda b\circ Sf)^\mathbbm{T}} & & TSb
 }$$
 commutes,
 \item for every natural transformation $\kappa:H\Rightarrow KS:A\rightarrow Y$ and every $a\in A$ the diagram
 $$\xymatrix{
 HSTa\ar[rr]^{\kappa^\mathbbm{S}Ta}\ar[d]_{H\lambda a} & & KSTa\ar[d]^{K\lambda a}\\
 HTSa\ar[rr]_{(K\lambda\cdot\kappa T)^\mathbbm{S}a} & & KTSa
 }$$
 conmutes.
\end{enumerate*}
\end{definicion}

And we get the proposition:

\begin{proposicion}
Let $A$ be a category. There is a bijective correspondence between distributive laws of $(S,\eta',\mu')$ over $(T,\eta,\mu)$, monads on $A$, and no-iteration distributive laws of the right extension system of $(S,\eta',(-)^\mathbbm{S})$ over the monad $(T,\eta,(-)^\mathbbm{T})$ in extensive form.
\end{proposicion}
\begin{proof}
Let $\lambda:ST\Rightarrow TS$ be a distributive law of $(S,\eta'\mu')$ over $(T,\eta,\mu)$, monads on $A$. Let $(S,\eta',(-)^\mathbbm{S})$ be the right extension system corresponding to the monad $(S,\eta'\mu')$ and let $(T,\eta,(-)^\mathbbm{T})$ be the monad in extensive form corresponding to $(T,\eta,\mu)$.

Axiom 1 in Definition~\ref{D:noitdistlawcat} trivially holds for the family of arrows $\lambda a:STa\rightarrow TSa$.

Now let $f:a\rightarrow Tb$ be an arrow in $A$. Consider the natural transformation
$$\xymatrix{
\uno\ar[rr]^a\ar[dr]_b\rrtwocell<\omit>{<3>f} & & A\\
 & A\ar[ur]_T &
}$$
and paste as in the proof of Proposition~\ref{P:bijdistlawnnoitdistlaw}. So Axiom 2 in Definition~\ref{D:noitdistlawcat} holds.

The proof for Axiom 3 is the same as the corresponding one in Proposition~\ref{P:bijdistlawnnoitdistlaw}.

Conversely, let $(S,\eta',(-)^\mathbbm{S})$ be a right extension system on $A$ and let $(T,\eta,(-)^\mathbbm{T})$ be a monad on $A$ in extensive form. Let $\lambda$ be a no-iteration distributive law of $(S,\eta',(-)^\mathbbm{S})$ over $(T,\eta,(-)^\mathbbm{T})$.

Compatibility of $\lambda$ with $\eta'$ and $\eta$ trivially holds.

Note that
$$\cat(\uno,A)(a,Tb)\cong A(a,Tb);$$
whence, if we denote, as $[-]^\mathbbm{T}$, the $\mathbbm{T}$-left extension operator induced by the monad $(T,\eta,\mu)$ induced by the monad $(T,\eta,(-)^\mathbbm{T})$ in extensive form, we have that
$$(-)^\mathbbm{T}_{a,b}:A(a,Tb)\rightarrow A(Ta, Tb)$$
is
$$[-]^\mathbbm{T}_{a,b}:\cat(\uno,A)(a,Tb)\rightarrow\cat(\uno,A)(Ta,Tb).$$
Therefore, by the proof for the compatibility of $\lambda$ with $\mu$ in Proposition~\ref{P:bijdistlawnnoitdistlaw}, we have the compatibility of $\lambda$ with $\mu$ in this case.

Let us see that $\lambda$ is a natural transformation. Let $h:c\rightarrow d$ be an arrow in $A$. Then,
\begin{align}
\lambda d\circ STh &= \lambda d\circ S(\eta d\circ h)^\mathbbm{T}\notag\\
                   &= (\lambda d\circ S\eta d\circ Sh)^\mathbbm{T}\circ\lambda c &\text{Definition~\ref{D:noitdistlawcat}(2)}\notag\\
                   &= (\eta Sd\circ Sh)^\mathbbm{T}\circ\lambda c &\text{Definition~\ref{D:noitdistlawcat}(1)}\notag\\
                   &= TSh\circ\lambda c\notag.
\end{align}

The compatibility of $\lambda$ with $\mu'$ follows from Proposition~\ref{P:bijdistlawnnoitdistlaw}.
\end{proof}

\begin{observacion}
Theorem~\ref{T:bijdistlawalg} and Proposition~\ref{P:bijdistlawnnoitdistlaw} give us a characterization in between. Let $K$ be a 2-category and $a\in K$. A no-iteration distributive law of the left extension system $(s,\eta',(-)^\mathbbm{s})$ on $a$ over the left extension system $(t,\eta,(-)^\mathbbm{t})$ on $a$ is an $\mathbbm{s}$-algebra $(ts,(-)^{\bm{\lambda}})$ such that
\begin{enumerate*}
 \item for every $\gamma:g\Rightarrow sh:x\rightarrow a$ and $\vartheta:f\Rightarrow tsg:x\rightarrow a$ the diagram
$$\xymatrix{
sf\ar[r]^{\vartheta^{\bm{\lambda}}}\ar[dr]_{(t\gamma^\mathbbm{s}\cdot\vartheta)^{\bm{\lambda}}} & tsg\ar[d]^{t\gamma^\mathbbm{s}} \\
& tsh
}$$
commutes,
 \item the diagram
$$\xymatrix{
& st\ar[dd]^\lambda \\
s\ar[ur]^{s\eta}\ar[dr]_{\eta s} & \\
& ts
}$$
commutes and
\item for every $\nu:u\Rightarrow tv:y\rightarrow a$ the diagram
$$\xymatrix{
 stu\ar[rr]^{s\nu^\mathbbm{t}}\ar[d]_{\lambda u} & & stv\ar[d]^{\lambda v}\\
 tsu\ar[rr]_{(\lambda v\cdot s\nu)^\mathbbm{t}} & & tsv
 }$$
commutes,
\end{enumerate*}
where $\lambda:=(t\eta')^{\bm{\lambda}}$ in 2 and 3.

In $K=\Cat$ this becomes: a no-iteration distributive law of the monad in extension form $(S,\eta',(-)^{\mathbbm{S}})$ over the monad in extension form $(T,\eta,(-)^{\mathbbm{T}})$, both on $A$, consists of
\begin{enumerate*}
 \item for every $a\in A$ an $\mathbbm{S}$-algebra $(TSa,(-)^{\bm{\lambda} a})$,
 \item for every $h:b\rightarrow Sa$, $(\eta Sa\circ h^\mathbbm{S})^\mathbbm{T}:(TSb,(-)^{\bm{\lambda} b})\rightarrow(TSa,(-)^{\bm{\lambda} a})$ is a morphism of $\mathbbm{S}$-algebras,
\end{enumerate*}
and such that
\begin{enumerate*}
\setcounter{enumi}{2}
 \item for every $a\in A$ the diagram
 $$\xymatrix{
 & STa\ar[dd]^{\lambda a} \\
 Sa\ar[ur]^{(\eta'Ta\circ\eta a)^\mathbbm{S}}\ar[dr]_{\eta Sa} & \\
 & TSa
 }$$
 commutes,
 \item for every $f:a\rightarrow Tb$ in $A$ the diagram
 $$\xymatrix{
 STa\ar[rrr]^{(\eta'Tb\circ f^\mathbbm{T})^\mathbbm{S}}\ar[d]_{\lambda a} & & & STb\ar[d]^{\lambda b} \\
 TSa\ar[rrr]_{(\lambda b\circ (\eta'Tb\circ f)^\mathbbm{S})^\mathbbm{T}} & & & TSb
 }$$
 commutes,
 where $\lambda a:=((\eta Sa\circ\eta'a)^\mathbbm{T})^{\bm{\lambda} a}$.
\end{enumerate*}
\end{observacion}

\end{document}